\documentclass{amsart}

\usepackage{amssymb,color}
\RequirePackage[colorlinks,citecolor=blue,urlcolor=blue]{hyperref}

\newtheorem{theorem}{Theorem}
\newtheorem{lemma}[theorem]{Lemma}
\newtheorem{corollary}[theorem]{Corollary}

\theoremstyle{definition}

\theoremstyle{remark}

\newcommand{\E}{\mathbb{E}}
\newcommand{\R}{\mathbb{R}}

\begin{document}

\textbf{Probability Theory/ \textit{Probabilit\'{e}s}}
 
\title[The Baum--Katz theorem
for sequences of pairwise independent random variables]{On the Baum--Katz theorem
	for sequences of pairwise independent random variables with regularly varying normalizing constants}


\author{L\^{e} V\v{a}n {Th\`{a}nh}}
\address{Department of Mathematics, Vinh University, 182 Le Duan, Vinh, Nghe An, Vietnam}
\email{levt@vinhuni.edu.vn}

\date{}

\begin{abstract}
This paper
proves the Baum--Katz theorem 
for sequences of pairwise independent identically distributed random variables
with general norming constants
under optimal moment conditions. 
The proof exploits some properties of slowly
varying functions and the de Bruijn conjugates, and uses the techniques developed by Rio (1995)
to avoid using the maximal type inequalities.
\end{abstract}

\maketitle

\section{Introduction and result}

Let $1\le p<2,\ \alpha p\ge 1$ and $\{X,X_n,n\ge1\}$ be a sequence of pairwise independent identically distributed (p.i.i.d.)
random variables. In this paper, by
using some results related to slowly varying functions and techniques developed by Rio \cite{rio1995vitesses},
we provide the necessary and sufficient conditions for
\begin{equation}\label{BK01}
\sum_{n}
n^{\alpha p-2} \mathbb{P}\left(\max_{1\le k\le n}\left|
\sum_{i=1}^{k}X_i\right|>\varepsilon n^{1/\alpha}\tilde{L}(n^{1/\alpha})\right)<\infty
\ \text{ for all } \ \varepsilon >0,
\end{equation}
where $\tilde{L}(\cdot)$ is 
the de Bruijn conjugate of a slowly varying function $L(\cdot)$. The result provides
the rate of convergence in the Marcinkiewicz--Zygmund strong law of large numbers (SLLN)  with regularly varying normalizing constants. 
When the random variables are i.i.d. with $\E(X)=0,\ \E(|X|^p)<\infty$, and $L(\cdot)=1$, \eqref{BK01} was
obtained by Baum and Katz \cite{baum1965convergence}.

The notion of regularly varying function can be found in Seneta \cite[Chapter 1]{seneta1976regularly}.
A real-valued function $R(\cdot )$ is said to be regularly varying with index of regular variation
$\rho\in\mathbb{R}$ if it is 
a positive and measurable function on $[A,\infty)$ for some $A> 0$, and for each $\lambda>0$,
\begin{equation*}\label{rv01}
\lim_{x\to\infty}\dfrac{R(\lambda x)}{R(x)}=\lambda^\rho.
\end{equation*}
A regularly varying function with the index of regular variation $\rho=0$ is called to be slowly varying.
It is well known that a function $R(\cdot )$ is regularly varying
with the index of regular variation $\rho$ if and only if it can be written in the form
\begin{equation*}\label{sv01}
R(x)=x^\rho L(x)
\end{equation*}
where $L(\cdot)$ is a slowly varying function (see, e.g., Seneta \cite[p. 2]{seneta1976regularly}).
Seneta \cite{seneta1973interpretation} (see also Lemma 1.3.2 in Bingham et al. \cite{bingham1989regular}) proved that if 
$L(\cdot)$ is a slowly varying function defined on $[A,\infty)$ for some $A> 0$,
then there exists $B\ge A$ such that $L(x)$ is bounded on every finite closed interval $[a,b]\subset [B,\infty)$.
Galambos and Seneta \cite[p. 111]{galambos1973regularly} showed that for any slowly varying function $L(x)$,
there exists
a differentiable slowly varying function $L_1(\cdot)$ defined on $[B,\infty)$ for some $B\ge A$
such that 
\begin{equation*}
\lim_{x\to\infty}\dfrac{L(x)}{L_{1}(x)}=1\ \text{ and }\ \lim_{x\to\infty}\dfrac{xL_{1}'(x)}{L_{1}(x)}=0.
\end{equation*}
Conversely, if  $L(\cdot)$ is a positive differentiable function satisfying 
\begin{equation}\label{sv02}
\lim_{x\to\infty}\dfrac{xL'(x)}{L(x)}=0,
\end{equation}
then $L(\cdot)$ is a slowly varying function.  
If $L(\cdot)$ be a differentiable slowly varying function 
satisfying \eqref{sv02}, then by direct calculations (see, e.g., Lemma 2.3 of Anh et al. \cite{anh2021marcinkiewicz}), 
we can show that for all $p>0$, there exists $B>0$ such that
\begin{equation}\label{sv03}
x^pL(x) \text{ is strictly increasing on } [B,\infty),\ x^{-p}L(x) \text{ is strictly decreasing on } [B,\infty). 
\end{equation}
Let $L(\cdot)$ be a slowly varying function. Then by Theorem 1.5.13 in Bingham et al. \cite{bingham1989regular},
there exists a slowly varying function $\tilde{L}(\cdot)$, unique up to asymptotic equivalence, satisfying
\begin{equation}\label{BGT1513}
\lim_{x\to\infty}L(x)\tilde{L}\left(xL(x)\right)=1\ \text{ and } \lim_{x\to\infty}\tilde{L}(x)L\left(x\tilde{L}(x)\right)=1.
\end{equation}
The function $\tilde{L}$ is called the de Bruijn conjugate of $L$, and $\left(L,\tilde{L}\right)$ is called a (slowly varying) conjugate pair (see, e.g., 
Bingham et al. \cite[p. 29]{bingham1989regular}).
By \cite[Proposition 1.5.14]{bingham1989regular}, if $\left(L,\tilde{L}\right)$ is a conjugate pair, then for $a,b,\alpha>0$, 
each of $\left(L(ax),\tilde{L}(bx)\right)$, $\left(aL(x),a^{-1}\tilde{L}(x)\right),$
$\left(\left(L(x^\alpha)\right)^{1/\alpha},(\tilde{L}(x^\alpha))^{1/\alpha}\right)$
is a conjugate pair.
Bojani\'{c} and Seneta \cite{bojanic1971slowly} (see also Theorem 2.3.3 and Corollary 2.3.4 in Bingham et al. \cite{bingham1989regular})
proved that if $L(\cdot)$ is a slowly varying function satisfying 
\begin{equation}
\lim_{x\to\infty}\left(\dfrac{L(\lambda_0 x)}{L(x)}-1\right)\log( L(x))=0,
\end{equation}
for some $\lambda_0>1$, then for all $\alpha\in\R$,
\begin{equation}\label{BGT2.3.4}
\lim_{x\to\infty}\dfrac{L(xL^\alpha(x))}{L(x)}=1,
\end{equation}
and therefore, we can choose (up to aymptotic equivalence) $\tilde{L}(x)=1/L(x)$.
Especially, if for some $\gamma\in\R$, $L(x)=\log^\gamma(x+2)$, $x\ge0$, then $\tilde{L}(x)=1/L(x)$.
For $\alpha,\beta>0$ and for $f(x)=x^{ \beta/\alpha}L^{1/\alpha}(x^\beta)$,
$g(x)=x^{\alpha/\beta}\tilde{L}^{\beta}(x^{\alpha})$, we have (see Theorem 1.5.12 and Proposition 1.5.15 in \cite{bingham1989regular})
\begin{equation}\label{sv05}
\lim_{x\to\infty}\dfrac{f(g(x))}{x}=\lim_{x\to\infty}\dfrac{g(f(x))}{x}=1.
\end{equation} 

Here and thereafter, for a slowly varying function $L(\cdot)$,
we denote the de Bruijn
conjugate of $L(\cdot)$
by $\tilde{L}(\cdot)$. 
We will assume, without loss of generality, that
$L(x)$ and $\tilde{L}(x)$ are both defined on $[0,\infty)$ and
differentiable on $[A,\infty)$ for some $A>0$.
The following theorem is the main result of this paper. 
The Marcinkiewicz--Zygmund SLLN with regularly varying normalizing constants was also
studied recently by Anh et al. \cite{anh2021marcinkiewicz}, where the proof is based on the Kolmogorov maximal inequality.

\begin{theorem}
\label{main}
	Let $1\le p<2$, and let
	$\{X,X_n, \, n \geq 1\}$ be a sequence of p.i.i.d. random variables, $L(\cdot)$ a slowly varying function defined on $[0,\infty)$.
	When $p=1$, we assume further that $L(x)\ge 1$ and is increasing on $[0,\infty)$.
	Then the 
	following statements are equivalent.
	\begin{itemize}
		\item[(i)] The random variable $X$ satisfies
		\begin{equation}\label{add007}
		\E(X)=0,\ \E\left(|X|^p L^p(|X|)\right)<\infty.
		\end{equation}
		\item[(ii)] For all $\alpha\ge 1/p$, we have
		\begin{equation}\label{sv003}
		\sum_{n= 1}^{\infty}
		n^{\alpha p-2}\mathbb{P}\left(\max_{1\le j\le n}\left|
		\sum_{i=1}^{j}X_i\right|>\varepsilon  n^{\alpha}{\tilde{L}}(n^{\alpha})\right)<\infty \text{ for  all } \varepsilon >0.
		\end{equation}
		\item[(iii)] The Marcinkiewicz--Zygmund-type SLLN
		\begin{equation}\label{add014}
		\begin{split}
		\lim_{n\to\infty}\dfrac{
			\max_{1\le k\le n}\left|\sum_{i=1}^{k}X_i\right|}{n^{1/p}{\tilde{L}}(n^{1/p})}=0\ \text{ almost surely (a.s.)}
		\end{split}
		\end{equation}
		holds.
		
	\end{itemize}
\end{theorem}

For a sequence of random variables which are pairwise independent but not identically distributed,  Cs\"{o}rg\H{o}
et al. \cite{csorgo1983strong} proved that the Kolmogorov condition alone does not ensure the SLLN.
On the case where the random variables $\{X,X_n,n\ge 1\}$ are p.i.i.d, Etemadi \cite{etemadi1981elementary} proved that the Kolmogorov SLLN holds
under moment condition $\E(|X|)<\infty$. For $\gamma>0$, Martikainen \cite{martikainen1995remark} proved that
\[\lim_{n\to\infty}\dfrac{
	\sum_{i=1}^{n}X_i}{n \log^{-\gamma}(n)}=0\ \text{ a.s.}\]
if and only if 	$\E(X)=0$ and $\E(|X|\log^{\gamma}(|X|+2)) <\infty$.
This is a special case of Theorem \ref{main} when $p=1$ and $L(x)\equiv \log^{\gamma}(x+2)$. For the case where $1<p<2$, Martikainen \cite{martikainen1995strong}
proved that if $\E(X)=0$ and $\E(|X|^p\log^r(|X|+1))<\infty$ for some $r>\max\{0,4p-6\}$, then the Marcinkiewicz--Zygmund SLLN holds.
By letting $L(x)\equiv 1$, we obtain the following corollary. When $\alpha=1/p$, this corollary was obtained
by Rio \cite{rio1995vitesses}.

\begin{corollary}\label{cor-01}
	Let $1\le p<2$, and let
	$\{X,X_n, \, n \geq 1\}$ be a sequence of p.i.i.d. random variables.
	Then the 
	following statements are equivalent.
	\begin{itemize}
		\item[(i)] The random variable $X$ satisfies
		\begin{equation*}
		\E(X)=0,\ \E\left(|X|^p)\right)<\infty.
		\end{equation*}
		\item[(ii)] For all $\alpha\ge 1/p$, we have
		\begin{equation*}
		\sum_{n= 1}^{\infty}
		n^{\alpha p-2}\mathbb{P}\left(\max_{1\le j\le n}\left|
		\sum_{i=1}^{j}X_i\right|>\varepsilon  n^{\alpha}\right)<\infty \text{ for  all } \varepsilon >0.
		\end{equation*}
		\item[(iii)] The Marcinkiewicz--Zygmund SLLN
		\begin{equation*}
		\begin{split}
		\lim_{n\to\infty}\dfrac{
			\max_{1\le k\le n}\left|\sum_{i=1}^{k}X_i\right|}{n^{1/p}}=0\ \text{ a.s.}
		\end{split}
		\end{equation*}
		holds.
	\end{itemize}
\end{corollary}

\section{Proof}
To prove the main result, we firstly introduce some preliminaries.
Through this paper, $C(\cdot)$, $C_1(\cdot), C_2(\cdot),...$ denote constants which depend only on variables appearing in the parentheses.
The first lemma is a direct consequence of Karamata's theorem (see \cite[Proposition 1.5.8]{bingham1989regular}).

\begin{lemma}\label{sv51}
	Let $a,b>1$, and $L(\cdot)$ be a differentiable slowly varying function defined on $[0,\infty)$.
	Then 
	\begin{equation*}\label{rem08}
	\begin{split}
	\sum_{k=1}^n a^kL(b^k) \le C_1(a,b)a^nL(b^n).
	\end{split}
	\end{equation*}
\end{lemma}

The following lemma gives simple criterions for $\E\left(|X|^p L^p(|X|)\right)<\infty$. When $\alpha=1/p$,
the equivalence of \eqref{81} and \eqref{82} was established by Anh et al. \cite[Proposition 2.6]{anh2021marcinkiewicz}.

\begin{lemma}\label{lemma_bound01}
	Let $p\ge 1,\ \alpha p\ge 1$, and $X$ be a random variable.
	Let $L(x)$ be a slowly varying function defined on $[0,\infty)$, and $b_n=n^{\alpha}\tilde{L}\left(n^{\alpha}\right)$, $n\ge 1$. 
	Assume that 
	$x^{1/\alpha}L^{1/\alpha}(x)$ and $x^{\alpha}\tilde{L}(x^{\alpha})$
	are 
strictly  
	increasing on $[A,\infty)$ for some $A>0$.
	Then the following statements are equivalent.
		\begin{equation}\label{81}
		\E\left(|X|^p L^p(|X|)\right)<\infty.
		\end{equation}

		\begin{equation}\label{82}
		\sum_{n=1}^\infty n^{\alpha p-1}\mathbb{P}(|X|>b_n)<\infty.
		\end{equation}	

		\begin{equation}\label{83}
		\begin{split}
		\sum_{n=1}^\infty 2^{n\alpha p}\mathbb{P}(b_{2^{n-1}}<|X|\le b_{2^n})<\infty.	
		\end{split}
		\end{equation}
\end{lemma}

\begin{proof}
	Let $f(x)=x^{1/\alpha} L^{1/\alpha}(x)$, $g(x)=x^{\alpha}\tilde{L}(x^{\alpha})$.
	By using \eqref{sv05} with $\beta=1$, we have 
	\begin{equation}\label{84}
	f(g(x))\sim g(f(x)) \sim x\text{ as }x\to\infty.
	\end{equation}
	Firstly, we will prove \eqref{81} is equivalent to \eqref{82}. For a non negative random variable $Y$ and $r>0$, $\E Y^r<\infty$ if and only if $\sum_{n=1}^\infty n^{r-1}\mathbb{P}(Y>n)<\infty$.
	Applying this, we have that 
	$\E\left(f^{\alpha p}(|X|)\right)<\infty$ if and only if
	\begin{equation}\label{r01}
	\sum_{n=1}^\infty n^{\alpha p-1} \mathbb{P}\left(f(|X|)>n\right)<\infty.
	\end{equation}
	Combining \eqref{84} with the assumption that 
	$f(x)$ and $g(x)$
	are strictly increasing on $[A,\infty)$, we see that \eqref{r01} is equivalent to
	\begin{equation*}\label{r02}
	\sum_{n=1}^\infty n^{\alpha p-1} \mathbb{P}\left(|X|>n^{\alpha}\tilde{L}(n^{\alpha})\right)<\infty.
	\end{equation*}
	The proof of the equivalence of \eqref{81} and \eqref{82} is completed. Now, we will prove \eqref{81} is equivalent to \eqref{83}.
	For $n$ large enough, on event $(b_{2^{n-1}}<|X|\le b_{2^n})$, we have
	\[f^{\alpha p}(b_{2^{n-1}})< f^{\alpha p}(|X|) \le f^{\alpha p}(b_{2^{n}}),\]
	or equivalently,
	\begin{equation}\label{85}
	\left(f(g(2^{n-1}))\right)^{\alpha p}< |X|^pL^p(|X|) \le \left(f(g(2^{n}))\right)^{\alpha p}.
	\end{equation}
	Combining \eqref{84} and \eqref{85}, we see that \eqref{83} is equivalent to \eqref{81}.
\end{proof}

\begin{proof}[Proof of Theorem \ref{main}]
	By the arguments leading to \eqref{sv02} and \eqref{sv03}, without loss of generality, we can assume that
	there exists a positive integer $A$ large enough such that 
	$x^{1/\alpha}L(x^{1/\alpha})$,
	$x^{\alpha}\tilde{L}(x^{\alpha})$ and $x^{p-1}L^p(x)$ (for $p>1$) are strictly increasing on $[A,\infty)$.
	
	Firstly, we prove the implication ((i) $\Rightarrow$ (ii)).
	It is easy to see that \eqref{sv003} is equivalent to
	\begin{equation}\label{sv006}
	\sum_{n= 1}^{\infty}
	2^{n(\alpha p-1)}\mathbb{P}\left(\max_{1\le j<2^n}\left|
	\sum_{i=1}^{j}X_i\right|>\varepsilon  2^{n\alpha}{\tilde{L}}\left(2^{n\alpha}\right)\right)<\infty \text{ for  all } \varepsilon >0.
	\end{equation}
	Assume that \eqref{add007} holds. For $n\ge 1$, set $b_n=n^{\alpha}{\tilde{L}}(n^{\alpha})$ and
	\[X_{i,n}=X_i\mathbf{1}(|X_i|\le b_n),\ 1\le i\le n.\]
	For all $\varepsilon>0$ and $n\ge 1$, we have
	\begin{equation}\label{sv005}
	\begin{split}
	&\mathbb{P}\left(\max_{1\le j<2^{n}}\left|\sum_{i=1}^j
	X_{i}\right|>\varepsilon  b_{2^{n}}\right)\le \mathbb{P}\left(\max_{1\le i<2^{n}}|X_i|>b_{2^{n}}\right)+\mathbb{P}\left(\max_{1\le j<2^{n}}\left|\sum_{i=1}^j  
	X_{i,2^{n}}\right|>\varepsilon  b_{2^{n}}\right)\\
	&\le \mathbb{P}\left(\max_{1\le i<2^{n}}|X_i|>b_{2^{n}}\right)+\mathbb{P}\left(\max_{1\le j<2^{n}}\left|\sum_{i=1}^j  
	(X_{i,2^{n}}-\E X_{i,2^{n}})\right|>\varepsilon
	b_{2^{n}}- 
	\sum_{i=1}^{2^{n}} \left|E( X_{i,2^{n}})\right|\right).
	\end{split}
	\end{equation}
	Since $\tilde{L}(\cdot)$ is slowly varying, $b_{2^{n+1}}\le 4^\alpha b_{2^{n}}$ for $n\ge n_0$ for some $n_0\ge A$.
	By the second half of \eqref{add007} and Lemma \ref{lemma_bound01}, we have
	\begin{equation}\label{sv007}
	\begin{split}
	\infty&>\sum_{j= 1}^{\infty}  j^{\alpha p-1}\mathbb{P}(4^\alpha|X|>b_j)=\sum_{j= 1}^{\infty} j^{\alpha p-2}\sum_{i=1}^{j}\mathbb{P}
	\left(4^\alpha|X_i|>b_j\right)\\
	&\ge \sum_{j= 1}^{\infty} j^{\alpha p-2}\mathbb{P}\left(\max_{1\le i\le j}
	4^\alpha|X_i|>b_j\right)= \sum_{n= 0}^{\infty} \sum_{j=2^{n}}^{2^{n+1}-1} j^{\alpha p-2}\mathbb{P}\left(\max_{1\le i\le j}
	4^\alpha|X_i|>b_j\right)\\
	&\ge \dfrac{1}{2}\sum_{n= 0}^{\infty} \sum_{j=2^{n}}^{2^{n+1}-1} 2^{n(\alpha p-2)}\mathbb{P}\left(\max_{1\le i< {2^{n}}}
	4^\alpha|X_i|>b_{j}\right)\\
	&\ge \dfrac{1}{2} \sum_{n=n_0}^{\infty} 2^{n(\alpha p-1)}\mathbb{P}\left(\max_{1\le i< {2^{n}}}
	|X_i|>b_{2^{n}}\right).
	\end{split}
	\end{equation}
	For $n\ge 1$, the first half of \eqref{add007} imply that
	\begin{equation}\label{sv37}
	\begin{split}
	\dfrac{|\sum_{i=1}^{n}\E(X_{i,n})|}{b_n}& \le \dfrac{\sum_{i=1}^{n}\left|\E X_{i}\mathbf{1}(|X_i|> b_n)\right|}{b_n}\\
	&\le \dfrac{n\E(|X|\mathbf{1}(|X|> b_n))}{b_n}.
	\end{split}
	\end{equation}
	For $n$ large enough
	and for $\omega \in (|X|>b_n)$, we have
	\begin{equation}\label{sv39}
	\begin{split}
	\dfrac{n}{b_n}&\le \dfrac{n^{(p-1)\alpha}\tilde{L}^{p-1}(n^{\alpha})}{\tilde{L}^{p}(n^{\alpha})}= \dfrac{\left(n^{\alpha}\tilde{L}(n^{\alpha})\right)^{p-1}L^p \left(n^{\alpha}\tilde{L}(n^{\alpha})\right)}{\tilde{L}^{p}(n^{\alpha})L^p 
		\left(n^{\alpha}\tilde{L}(n^{\alpha})\right)}\\
	&\le 2 b_{n}^{p-1}L^p(b_n) \le  2|X(\omega)|^{p-1} L^p(|X(\omega)|),
	\end{split}
	\end{equation}
	where we have applied the second half of \eqref{BGT1513} in the second inequality and the monotonicity of $x^{p-1}L^p(x)$ in the third inequality.
	It follows from \eqref{sv39} and the second half of \eqref{add007} that
	\begin{equation}\label{sv41}
	\begin{split}
	\dfrac{n\left|\E(|X|\mathbf{1}(|X|> b_n))\right|}{b_n}&
	\le 2\E\left(|X|^p L^p(|X|)\mathbf{1}\left(|X|> b_n\right)\right)\to 0 \text{ as } n\to \infty.
	\end{split}
	\end{equation}
	From  
	\eqref{sv005}, \eqref{sv007}, \eqref{sv37} and \eqref{sv41}, the proof of \eqref{sv006} will be completed if we can show that
	\begin{equation}\label{21}
	\sum_{n=1}^{\infty} 2^{n(\alpha p-1)} \mathbb{P}\left(\max_{1\le j< 2^n}\left|\sum_{i=1}^j  
	(X_{i,2^n}-\E X_{i,2^n})\right|\ge \varepsilon b_{2^{n-1}} \right)<\infty \text{ for all }\varepsilon>0.
	\end{equation}
	For $m\ge 0,$ set $S_{0,m}=0$ and
	\[S_{j,m}=\sum_{i=1}^j (X_{i,2^m}-\E X_{i,2^m}),\ j\ge 1.\]
	Now, we use techniques developed by Rio \cite{rio1995vitesses}.
	For $1\le j<2^n$ and for $0\le m\le n$, let $k=\lfloor j/2^m\rfloor $ be the greatest integer which is
	less than or equal to $j/2^m$. Then $0\le k<2^{n-m}$ and $k 2^m\le j< (k+1)2^m$. Let $j_m = k 2 ^m$, then
	\begin{equation}\label{23}
	S_{j,n}=\sum_{m=1}^{n}(S_{j_{m-1},m-1}-S_{j_{m},{m-1}})+\sum_{m=1}^{n}(S_{j,m}-S_{j,m-1}-S_{j_m,m}+S_{j_m,m-1}),\  1\le j<2^n.
	\end{equation}
and
	\begin{equation}\label{24}
	\begin{split}
	\left|S_{j,m}-S_{j,m-1}-S_{j_m,m}+S_{j_m,m-1}\right|& \le \sum_{i=j_m+1}^{j_m+2^m}
	\left(\left|X_ {i, 2^m}-X_ {i, 2^{m-1}}\right|+\E\left|X_ {i, 2^m}-X_ {i, 2^{m-1}}\right|\right).
	\end{split}
	\end{equation}
Set
	\[Y_ {i, m} =\left|X_ {i, 2^m}-X_{i, 2^{m-1}}\right|-\E\left(\left|X_{i, 2^m}-X_{i, 2^{m-1}}\right|\right),\ m\ge1,i\ge1.\]
It follows from \eqref{24} that
	\begin{equation}\label{27}
	\begin{split}
	\left|S_{j,m}-S_{j,m-1}-S_{j_m,m}+S_{j_m,m-1}\right|& \le \sum_{i=j_m+1}^{j_m+2^m}Y_{i, m}+2^{m+1}
	\E\left(\left|X_ {1, 2^m}-X_ {1, 2^{m-1}}\right|\right).
	\end{split}
	\end{equation}
	By the definition of $j_m$, we have either $j_{m-1}=j_m$ or $j_{m-1}=j_m+2^{m-1}$. Therefore
	\begin{equation}\label{27a}
	\left|S_{j_{m-1},m-1}-S_{j_{m},{m-1}}\right|\le \left|\sum_{i=j_m+1}^{j_m+2^{m-1}}\left(X_ {i, 2^{m-1}}-\E(X_{i, 2^{m-1}})\right)\right|.
	\end{equation}
	Combining \eqref{23}, \eqref{27} and \eqref{27a}, we have
	\begin{equation}\label{28}
	\begin{split}
	\max_{1\le j< 2^n}\left|S_{j,n}\right|
	&\le \sum_{m=1}^n \max_{0\le k<2^{n-m}}\left|\sum_{i=k2^m+1}^{k2^m+2^{m-1}}\left(X_{i, 2^{m-1}}-\E(X_{i, 2^{m-1}})\right)\right| \\
	&\quad +\sum_{m=1}^n \max_{0\le k<2^{n-m}}\left|\sum_{i=k2^m+1}^{(k+1)2^m}Y_{i,m}\right|+\sum_{m=1}^n 2^{m+1}
	\E\left(\left|X_ {1, 2^m}-X_ {1, 2^{m-1}}\right|\right).
	\end{split}
	\end{equation}
	It follows from \eqref{sv41} that
	\begin{equation}\label{33}
	\begin{split}
	\dfrac{2^{m}\E\left(|X|\textbf{1}\left(|X|> b_{2^{m-1}}\right)\right)}{b_{2^{m}}}\to 0\ \text{ as } \ m\to\infty.
	\end{split}
	\end{equation}
From \eqref{33} and  Lemma \ref{sv51}, we can apply Toeplitz's lemma and conclude that
 	\begin{equation}\label{34}
	\begin{split}
\lim_{n\to \infty}\dfrac{\sum_{m=1}^n 2^{m}\E\left(|X|\textbf{1}\left(|X|> b_{2^{m-1}}\right)\right)}{b_{2^n}}=0.
	\end{split}
	\end{equation}
By using
\[\E\left(\left|X_ {1, 2^m}-X_ {1, 2^{m-1}}\right|\right)\le  \E\left(|X|\textbf{1}\left(|X|>b_{2^{m-1}}\right)\right),\ m\ge A, \]
we have from \eqref{34} that
\begin{equation}\label{35}
\begin{split}
\lim_{n\to \infty}\dfrac{\sum_{m=1}^n 2^{m+1}\E\left(\left|X_ {1, 2^m}-X_ {1, 2^{m-1}}\right|\right)}{b_{2^n}}=0.
\end{split}
\end{equation}
	Let $\varepsilon_1>0$ be arbitrary, and let
	$a$ and $b$ be positive constants satisfying
	\[\alpha p/2<a<\alpha,\ a+b=\alpha.\]
	For $n\ge 1$, $0\le m\le n$,  let $\lambda_{m,n}=\varepsilon_{1}2^{bm}  2^{an}\tilde{L}(2^{n\alpha})$.
	Then
	\begin{equation}\label{30}
	\begin{split}
	\sum_{m=1}^n\lambda_{m,n}&=\varepsilon_1 2^{an}\tilde{L}(2^{n\alpha})\sum_{m=1}^n 2^{mb}\\
	&=\varepsilon_1 2^{an}2^{b}\tilde{L}(2^{n\alpha})\dfrac{2^{bn}-1}{2^{b}-1}\le \dfrac{ 2^b \varepsilon_1 b_{2^{n}}}{2^b-1}:=C_1(b)\varepsilon_1 b_{2^{n}}.
	\end{split}
	\end{equation}
	By \eqref{30} and Chebyshev's inequality, we have
	\begin{equation}\label{26}
	\begin{split}
	&\mathbb{P}\left(\sum_{m=1}^n \max_{0\le k<2^{n-m}}\left|\sum_{i=k2^m+1}^{(k+1)2^m}Y_{i,m}\right|\ge C_1(b)\varepsilon_1 b_{2^{n}}\right)\\
	&\le \sum_{m=1}^n\mathbb{P}\left(\max_{0\le k<2^{n-m}}\left|\sum_{i=k2^m+1}^{(k+1)2^m}Y_{i,m}\right|\ge \lambda_{m,n}\right)\\
	&\le \sum_{m=1}^n  \lambda_{m,n}^{-2} \E\left(\max_{0\le k<2^{n-m}}\left|\sum_{i=k2^m+1}^{(k+1)2^m}Y_{i,m}\right|\right)^2\\
	&\le \sum_{m=1}^n \lambda_{m,n}^{-2} \sum_{k=0}^{2^{n-m}-1}\E\left(\sum_{i=k2^m+1}^{(k+1)2^m}Y_{i,m}\right)^2\\
	&\le  \sum_{m=1}^n 2^n \lambda_{m,n}^{-2}\E\left(|X_{i,2^m}-X_{i,2^{m-1}}|^2\right)\\
	&\le  \sum_{m=1}^n 2^{n+1} \lambda_{m,n}^{-2}\left(\E (X_{i,2^{m}}^2)+\E(X_{i,2^{m-1}}^2)\right)\\
	&=  \sum_{m=1}^n 2^{n+1} \lambda_{m,n}^{-2}\left(\E\left(X^2\mathbf{1} (|X|\le b_{2^{m}})\right)+\E\left(X^2\mathbf{1} (|X|\le b_{2^{m-1}})\right)\right),
	\end{split}
	\end{equation}
	and
	\begin{equation}\label{26a}
	\begin{split}
	&\mathbb{P}\left(\sum_{m=1}^n \max_{0\le k<2^{n-m}}\left|\sum_{i=k2^m+1}^{k2^m+2^{m-1}}\left(X_ {i, 2^{m-1}}-\E(X_ {i, 2^{m-1}})\right)\right| \ge C_1(b)\varepsilon_1 b_{2^{n}}\right)\\
	&\le \sum_{m=1}^n\mathbb{P}\left(\max_{0\le k<2^{n-m}}\left|\sum_{i=k2^m+1}^{k2^m+2^{m-1}}\left(X_ {i, 2^{m-1}}-\E(X_ {i, 2^{m-1}})\right)\right| \ge \lambda_{m,n}\right)\\
	&\le \sum_{m=1}^n  \lambda_{m,n}^{-2} \E\left(\max_{0\le k<2^{n-m}}\left|\sum_{i=k2^m+1}^{k2^m+2^{m-1}}\left(X_ {i, 2^{m-1}}-\E(X_ {i, 2^{m-1}})\right)\right| \right)^2\\
	&\le \sum_{m=1}^n \lambda_{m,n}^{-2} \sum_{k=0}^{2^{n-m}-1}\E\left(\sum_{i=k2^m+1}^{k2^m+2^{m-1}}\left(X_ {i, 2^{m-1}}-\E(X_ {i, 2^{m-1}})\right)\right)^2\\
	&\le  \sum_{m=1}^n 2^n \lambda_{m,n}^{-2}\E\left(X_ {1, 2^{m-1}}^2\right)=  \sum_{m=1}^n 2^n \lambda_{m,n}^{-2}\E\left(X^2\mathbf{1} (|X|\le b_{2^{m-1}})\right).
	\end{split}
	\end{equation}
Since $\alpha p<2a$ and $1\le p<2$, elementary calculations show that
\begin{equation}\label{26b}
\sum_{n=1}^{\infty}2^{n(\alpha p-1)}\sum_{m=1}^{n}2^n\lambda_{m, n}^{-2}\le \dfrac{1}{\varepsilon_{1}^{2}(4^b-1)}\sum_{n=1}^{\infty}2^{n(\alpha p-2a)}\tilde{L}^{-2}(2^{n\alpha})<\infty.
\end{equation}
We recall that $b_n$ is strictly increasing on $[A,\infty)$.
From \eqref{28}, \eqref{35}, and \eqref{26}--\eqref{26b}, the proof of \eqref{21} is completed if we can show that
	\begin{equation}\label{38}
	I:=\sum_{n=A}^\infty 2^{n(\alpha p-1)} \sum_{m=A}^n 2^n \lambda_{m,n}^{-2}\sum_ {k = A}^m b_{2^{k}}^2 \mathbb{P}\left(b_{2^{k-1}}<|X|\le b_{2^k}\right)<\infty.
	\end{equation}
By using the second half of \eqref{add007}, Lemmas \ref{sv51}--\ref{lemma_bound01}, and definition of $a$ and $b$, we have
	\begin{equation*}\label{26c}
	\begin{split}
	I&= \sum_{n=A}^{\infty}2^{n(\alpha p-1)} \varepsilon_{1}^{-2} 2^{n (1-2a)}\tilde{L}^{-2}(2^{n\alpha})\left(\sum_{m=A}^n 2^{- 2mb} \sum_ {k = A}^m b_{2^{k}}^2 \mathbb{P}\left(b_{2^{k-1}}<|X|\le b_{2^k}\right)\right)\\
	&\le C_2(b) \varepsilon_{1}^{-2} \sum_{n=A}^{\infty} 2^{n (\alpha p-2a)}\tilde{L}^{-2}(2^{n\alpha})\sum_ {k = A}^n 2^{- 2kb} b_{2^{k}}^2 \mathbb{P}\left(b_{2^{k-1}}<|X|\le b_{2^k}\right)\\
	&\le C_2(b) \varepsilon_{1}^{-2} \sum_{k=A}^{\infty} \left(\sum_ {n=k}^\infty 2^{n (\alpha p-2a)}\tilde{L}^{-2}(2^{n\alpha})\right) 2^{- 2kb} b_{2^{k}}^2 \mathbb{P}\left(b_{2^{k-1}}<|X|\le b_{2^k}\right)\\
	&\le C(\alpha,a,b,p) \varepsilon_{1}^{-2} \sum_{k=A}^{\infty} 2^{k(\alpha p-2a)}\tilde{L}^{-2}(2^{k \alpha}) 2^{- 2kb} b_{2^{k}}^2 \mathbb{P}\left(b_{2^{k-1}}<|X|\le b_{2^k}\right)\\
	&= C(\alpha,a,b,p) \varepsilon_{1}^{-2} \sum_{k=A}^{\infty}  2^{k\alpha p} \mathbb{P}\left(b_{2^{k-1}}<|X|\le b_{2^k}\right)<\infty
	\end{split}
	\end{equation*}
	thereby proving \eqref{38}. The proof of the implication ((i) $\Rightarrow$ (ii)) is completed.
	
	By choosing $\alpha=1/p$, we have the implication ((ii) $\Rightarrow$ (iii)). The proof of the implication ((iii) $\Rightarrow$ (i))
	follows from the Borel--Cantelli lemma for pairwise independent events
	and Lemma \ref{lemma_bound01} (see the proof of Theorem 3.1 in \cite{anh2021marcinkiewicz}).
	
\end{proof}

\bibliographystyle{amsplain}
\bibliography{mybib}
\end{document}